\documentclass[journal]{IEEEtran}
\usepackage{amsmath,amsfonts}
\usepackage{algorithmic}
\usepackage{algorithm}
\usepackage{array}
\usepackage{textcomp}
\usepackage{stfloats}
\usepackage{url}
\usepackage{verbatim}
\usepackage{graphicx}
\usepackage{cite}
\usepackage{amsthm}
\usepackage{amssymb}

\newtheorem{theorem}{Theorem}
\newtheorem*{theorem*}{Theorem}
\newtheorem{corollary}{Corollary}
\newtheorem{lemma}{Lemma}
\newtheorem{proposition}{Proposition}

\newtheorem{remark}{Remark}

\def\wt{\widetilde}
\def\R{\mathbb R}
\def\E{\mathbb E}
\def\Pb{\mathbb P}

\def\S{\mathbb S}
\def\X{\mathbb X}

\def\1{\mathbf 1}  

\newcommand{\Nc}{\mathcal N}
\newcommand{\Cc}{\mathcal C}
\newcommand{\Ic}{\mathcal I}
\newcommand{\Tc}{\mathcal T}
\newcommand{\Fc}{\mathcal F}
\newcommand{\Bc}{\mathcal B}

\begin{document}

\title{Retransmission performance in a\\ stochastic geometric cellular network model}

\author{Ingemar Kaj, Taisia Morozova}
\date{\today}

\maketitle

\begin{abstract}

  Suppose sender-receiver transmission links in a downlink network at given
  data rate are subject to fading, path-loss and inter-cell
  interference, and that transmissions either pass, suffer loss, or
  incur retransmission delay.  We introduce a method to obtain the
  average activity level of the system required for handling the buffered work
  and from this derive the resulting coverage probability and key
  performance measures.  The technique involves a family of
  stationary buffer distributions which is used to solving iteratively 
  a nonlinear balance equation for the unknown busy-link probability and
  then identifying throughput, loss probability and delay.
  The results allow for straightforward numerical investigation of
  performance indicators, are in special cases explicit, and may be
  easily used to study the trade-off between reliability, latency, and
  data rate.

\end{abstract}
  
\section{Introduction}

The recent letter \cite{B2022} highlights the importance of analyzing
spatial and temporal stationarity in queues, and points to the need
for a tractable computational setting for certain classes of queueing
networks including cellular networks for wireless communication.  We
consider a simplified slotted time Poisson-Voronoi type stochastic
model of a wireless communication system and analyze conditions under
which such systems may operate long term, hence addressing some of the
challenges indicated in \cite{B2022}.
Our particular focus is 
towards highly reliable
controlled delay systems where
emitter nodes store failed transmissions in buffers and retransmit. In
this sense we aim at identifying a generic stochastic model set-up to
assist investigation of long-run ultrareliable and low-latency
communication for 5G wireless networks, \cite{BDP2018}.

Stochastic geometry provides widely used tools and techniques for
modeling and performance analysis of wireless communication systems.
Random point processes are used to represent spatially deployed
transmission nodes while additional random markings model spatial and
temporal features of signal transmission links. Nodes which transmit
simultaneously on a common channel cause interference between the
actively involved links. Hence, within the network, competition
emerges for shared capacity resources, leading to a number of issues
related to this type of modeling, about system stability, stationarity
in time and space, loss versus delay, etc.  The use of planar Poisson
point processes and stochastic geometry to study radio terminals can
be traced back at least to \cite{TK1984}, followed by
e.g.\ \cite{S1992} and \cite{IHV1998}.
Comprehensive and rigorous presentations of stochastic geometry and
its applications for wireless networks are covered in \cite{BB2009}
and \cite{H2012}. Some of the more recent contributions relying on
stochastic geometric methods have analyzed downlink cellular networks
\cite{ABG2011}, uplink cellular networks \cite{NDA2013}, interference
alignment \cite{GWF2016}, D2D networks \cite{SB2017,YLQ2016}, and
performance of full duplex systems \cite{LPH2018}.  A few studies
treat elements of retransmission or buffer mechanisms using a
stochastic geometric approach; in \cite{YQ2018} each base station is
equipped with an infinite size buffer and \cite{NMH2014} considers
cooperating base stations and a setting where data is successfully
transmitted if at most one retransmission is required.


\subsection{Poisson-Voronoi stochastic geometric models}

We 
consider a wireless network 
with two types of nodes which represent user equipment (UE) and base
stations (BS) 
and adhere to the simplified but
common approach that the placement of nodes is initially random in the
spatial domain but static in time.
A transmission link consists of an
emitter-receiver pair and connects a BS node with a UE node.  It is
convenient to simply place the nodes in all of $\R^d$, $d\ge 1$, to
avoid
boundary effects, and make the assumption of independence
between disjoint spatial regions in the sense of the Poisson point process on
$\R^d$.  The transmission
of signals between pairs of the nodes is further assumed to be
synchronized and slotted in time. In a fixed slot each sender attempts
to transmit the equivalent of one symbol.  The emergence of
emitter-receiver pairs and hence the pattern of signal transmit
connections will rely on the principle of shortest distance measured
from a UE to the nearest BS.  As a consequence, the service domains of
the BS-nodes form exactly as the random spatial cells known as the
Poisson-Voronoi tesselation.  In this framework, a number of detailed
assumptions are required in order to characterize how transmission
links emerge, whether transmitted signals are successfully received or
not, and how the performance of the system might be inferred.  Various
examples and research directions are discussed in the general
literature sources listed above.

Our 
present approach 
combines some of the typical and well-established model assumptions
and techniques with newly developed tools.  Thus, following a
tradition in wireless network modeling, signals are subject to
Rayleigh fast fading and path loss along a radially decreasing
attenuation function (either bounded or singular at
zero-distance). Moreover, simultaneous random fading variables are
assumed spatially independent and interference caused by shared
communication links is treated as added noise. For now we restrict to
downlink traffic for which the BS is the transmitter and the UE the receiver.
Among novel tools we introduce a counting method based on the
traditional iterative Lindley equation method for recording
retransmission requests as well as losses at the sender nodes.  This
allows us to present a consistent network model where failed
transmission attempts are either stored in a buffer and retransmitted
after some queueing delay, or lost due to lack of available buffer
space.  The main contribution is a technique to determine for given link
data rate input whether the system is able to process the work over
time and, if so, quantify the resulting loss and delay.

Briefly, we assume that new traffic arrives randomly to the system,
such that in any given time slot each designated transmission link is
requested to handle a new signal with probability $p$, $0\le p\le 1$,
independently between slots and between emitters.  In response, the
system needs to find a mode of operation in which a typical emitter is
active with some probability $q$, $p\le q\le 1$, the {\it busy-link
  probability}, which is large enough to allow new signal
transmissions as well as retransmissions in case of failed attempts.
Interference in our setup occurs between cells, not within cells, and
we consider the successful coverage of signals to be determined
by a meanfield criteria of the type $\mathrm{SINR}>T$, where
$\mathrm{SINR}$ is a suitable signal to interference-and-noise ratio
and $T>0$ a fixed threshold value.
The results are parameterized by an integer $K$ which represents the
maximal buffer size.
For finite $K$ and given $p$,
we use the
{\it coverage probability} $V_T(q)=\Pb(\mathrm{SINR}>T)$ to identify
the stationary buffer distribution, which in particular involves
solving a specific balance equation to obtain $q$ as function of
$p$. From this we obtain the buffer size distribution and efficient
means to quantify loss and delay.  The simplest case $K=0$ is a
pure-loss model with no retransmission and hence $q=p$.  The main
results in this work apply to the finite buffer model with $K\ge 1$.
which allows studying the trade-off between smaller loss and longer
delay as the buffer size increases.  For the idealized case of an
infinite buffer and hence no-loss, we identify a region of
stability. There exists a critical $p_c$, such that for traffic rates
$p<p_c$, there is a busy link probability $q$, $p<q<1$, for which the
wireless network model can operate in a steady-state sustained by a
geometric buffer size distribution.

\section{Network model}

\subsection{Voronoi cells acting as servers}

We let $\Phi_\mathrm{UE}=\{x_i\}$ and
$\Phi_\mathrm{BS}=\{x_j\}$ denote two independent Poisson point
processes in $\R^d$ with intensity measures $\lambda_0\,dx$ and
$\lambda_1\,dx$,  respectively, 
representing the locations of nonmobile user equipment and
locations of base stations. 
Formally, we consider a generic probability space $(\Omega,\Fc,\Pb)$
and define the point processes $\Phi$ ($\Phi_\mathrm{BS}$ or
$\Phi_\mathrm{UE}$) as random counting measures on $(\R^d,\Bc)$, where
$\Bc$ 
denotes the 
Borel sets on
$\R^d$. Hence, $\Phi: \Omega\times \Bc\mapsto \{0,1,\dots,+\infty\}$
is a mapping such that $\omega\mapsto \Phi(\omega,B)$ is an (extended)
integer-valued random variable for each $B\in\Bc$ and
$B\mapsto \Phi(\omega,B)$ is a Borel measure on $\R^d$ for each
outcome $\omega\in\Omega$. Dropping $\omega$ in notation we also write
$\Phi(B)=\int_B \Phi(dx)$ and, for integrable functions
$f$ on $\R^d$, 
$\int f(x)\,\Phi(dx)$.

We consider the Voronoi
tessellation formed by the Poisson points $\Phi_\mathrm{BS}$ as nodal
center points for Voronoi cells $\{\Cc_j, x_j\in \Phi_\mathrm{BS}\}$,
see e.g.\ \cite{SKMC2013} for the theory of
Poisson-Voronoi tessellations in $\R^d$. For example, the
volume $|\Cc_j|$ of each Voronoi cell has expected value
$\E[|\Cc_j|]=1/\lambda_1$.   With reference to the Palm measure
associated with $\Phi_\mathrm{BS}$ we assume 
first that one base station is located at the origin $0\in \R^d$,
and we let $\Cc$ be the fixed cell which contains the origin.
We consider the random sum  $\int_\Cc f(x)\,\Phi_\mathrm{UE}(dx)=
\sum_{x_i\in\Phi_\mathrm{UE} \cap \Cc} f(x_i)$ over all users within
the fixed cell
and recall that the
expected value of such sums evaluates as
\begin{equation}\label{eq:cellexpectation}
  \E\Big[ 
    \int_\Cc f(x)\,\Phi_\mathrm{UE}(dx)|\Phi_\mathrm{BS}  \Big]=
  \lambda_0 \int_{\R^d} f(x)\, e^{ -\lambda_1 |B(0,1)|
    |x|^d}\,dx
\end{equation}
\cite{FZ1996}, $d=2$. Here, $f$ is an integrable function, $|x|$
is euclidean distance, and $|B(0,1)|=\pi^{d/2}/\Gamma(1+d/2)$ is the
volume of the unit ball in $\R^d$.
For $f\equiv 1$ we obtain the
expected number of users
in a typical Voronoi cell as
\begin{equation}\label{eq:nmbusers}
w=  \E[\Phi_\mathrm{UE}(\Cc)|\Phi_\mathrm{BS}]
  =   \lambda_0/\lambda_1,
\end{equation}
which we may also think of as the bandwidth of the cell.
The ratio of (\ref{eq:cellexpectation}) and (\ref{eq:nmbusers}) is the
cell-average 
\begin{align}\label{eq:cell-average}
  \frac{\E[\langle \Phi_\mathrm{UE},f\cdot \1_\Cc\rangle|\Phi_\mathrm{BS}]}
  { \E[\Phi_\mathrm{UE}(\Cc)|\Phi_\mathrm{BS}] }
=  \lambda_1 \int_{\R^d} f(x)\, e^{ -\lambda_1 |B(0,1)| |x|^d}\,dx.
\end{align}
Here,  we may define a random variable $X$, which
represents the locations $(0,X)$ of the nodes in  a randomly
chosen transmission link within the cell $\Cc$, by the distribution
function
\[
  \Pb_\Cc(X\in B) =\lambda_1 \int_B  e^{ -\lambda_1 |B(0,1)|
    |x|^d}\,dx,\quad B\in \Bc.
\]
For a radial function, $f(x)=f(|x|)$, 
\begin{align}    \nonumber
 &\E_\Cc[f(X)]  =
  \lambda_1 \int_{\R^d} f(|x|)\, e^{ -\lambda_1 |B(0,1)| |x|^d}\,dx\\
  &\quad = \lambda_1 |B(0,1)|\int_0^\infty f(u)\,  e^{-\lambda_1|B(0,1)|u^d}du^{d-1}\,du,
\label{eq:radiusdistribution}
\end{align}
by which we recover the familiar result that averaging over Voronoi
cells is the same as integrating over a Rayleigh distributed radius
distance, compare e.g.\ \cite{NDA2013}. For convenience using our approach, we
rewrite (\ref{eq:radiusdistribution}) as follows.  Let
$\ell_0=(\lambda_1|B(0,1)|)^{-1/d}$ and associate with the Voronoi
cell $\Cc$ the corresponding Euclidean ball $B(0,\ell_0)$ in $\R^d$ of
radius $\ell_0$ and preserved expected volume
$|B(0,\ell_0)|=1/\lambda_1=\E[|\Cc|]$. Then for $\ell\ge 0$,
$\Pb_\Cc(|X|\le \ell) =1-e^{(\ell/\ell_0)^d}$ and
\begin{align}\label{eq:linkdistribution}
  \E_\Cc[f(|X|)] =   \int_0^\infty f(v^{1/d}\ell_0)\,  e^{-v}\,dv.
\end{align}
In practice, to evaluate key performance measures such as the coverage
probability we want to apply (\ref{eq:radiusdistribution}) or 
(\ref{eq:linkdistribution}) 
to functions $H(X,\Phi_X)$
of the link
$X$ and, given $X$, of the collection $\Phi_X$ of additional
nodes in surrounding cells which interfer with the link in $\Cc$.
To identify such functionals $H$ of interest and building further on the
point of view of considering the cells as servers of a queueing
network, we consider next the signal to noise and interference ratio for
the downlink
traffic scenario.

\subsection{Signal to interference and noise ratio}

Given fixed Poisson nodes and a Voronoi cell $\Cc$ we
consider a random transmission link consisting of an emitter-receiver
pair $(0,X)$ with distribution as introduced above.
By translation invariance of the Poisson
processess we place the receiver UE at the origin $0\in \R^d$ and the
emitter BS at $X\in \R^d$.
Downlink traffic from BS to UE nodes is subject to the interference caused by
base stations serving users in surrounding cells. The impact
in this situation
of the interference field on performance
in a
pure-loss system is analyzed in detail in \cite{ABG2011}. Our approach runs
in parallel but takes into account the case of failed attempts where
the signal packet is placed in a buffer and marked for
retransmission.
Each base station connects to all users nearer
to this BS than to any other service node.  Over the round of a single
time slot we assume that the service rate of the BS is proportional 
to the number of UE nodes in the cell, and that each UE requires a link 
to be set up with probability $p$, sufficiently small.  Hence it is assumed 
that the BS picks one UE uniformly located in the cell with probability
$\rho=p \lambda_0/\lambda_1<1$, 
the average measure of service load per slot.



To obtain explicit and computationally tractable
results we 
restrict 
to independent Rayleigh fading
signals. Hence, the signal power of an active transmitter in a given
slot is subject to fast fading determined by an exponential random
variable $S$ with mean $1/\mu$.  On the other hand, to balance the
tradeoff between generality and tractability, we allow 
pathloss mechanisms in a parametrized class of radial,
decreasing attenuation functions $a$, either of the standard shape
singular at the origin, $a(r)\sim r^{-\beta}$, $\beta>d$, or bounded
at the origin.  Our assumptions conform to a {\it standard stochastic
  scenario} with omni-directional path loss \cite[Section 2.3]{BB2009}.

Conditionally, given the cell $\Cc$ containing the origin and the
position of a target link $(0,X)=(0,x)$ between the emitter at $x$
and the receiver at $0$, the signal power reaching the
destined receiver is $Sa(x)$. Other transmitters which are active
during the same slot may inflict interference. A signal transmission
is successful in case a connection between the paired BS and UE is
established and the signal to interference and noise ratio,
$\mbox{SINR}$,
\[   
  \mbox{SINR(x)}= \frac{Sa(x)}{I(0,x)+\sigma^2}\,
  \1_{\{\mbox{link $(0,x)$ active}\}},
\]
exceeds a preset threshold value $T$.  Here, $\sigma^2\ge 0$
represents background noise (deterministic, for simplicity) and
$I(0,x)$ is the random interference power aggregated over all
simultaneous transmissions in the same slot.  In our scenario random
node locations are set at an initial time slot $n=0$. During each
subsequent slot $n\ge 1$ a new ratio SINR$_n(x)$ is generated at the
chosen link based on activity status and signal power of the link
itself as well as that of interfering links in the surrounding cells.
Now, independently for each slot $n\ge 1$, let $(A_n,S_n)$ be
independent random variables on the probability space
$(\Omega,\Fc,\Pb)$, such that $A_n\sim \mathrm{Bin}(1,p)$ is an
indicator variable for the event that the target link has a new signal
to transmit and $S_n\sim \mathrm{Exp}(\mu)$ is the associated signal
power.  Similarly, for each slot we consider the points $\Phi_x=
\Phi_\mathrm{BS}\setminus\{x\}$ of potential
emitter positions of intercellular interference links,
and let $\{(A_n^k,S_n^k)_{n\ge 1}, x_k\in\Phi_x\}$ be sequences of
independent marks on $(\Omega,\Fc,\Pb)$ associated with activity level
and signal strength of the interfering links.
The interference at the receiver is caused by the cumulative signal
power from other active emitters, so that in slot $n$
\begin{equation}\label{eq:sinr_general}
  I_n(0,x)=\sum_{x_k\in \Phi_x} S_n^ka(x_k)\, \1_{\{\mbox{emitter at
      $x_k$ active}\}}.
\end{equation}
To identify the intensity measure of $\Phi_x$ we
note that 
other BS's occur with intensity $\lambda_1$ each with
$\lambda_0/\lambda_1$ potential connections, 
and must be
located at a distance farther away from the origin than $|x|$. Hence
the thinned spatial intensity of the conditioned Poisson interference
field $\Phi_x$ is
\begin{equation}\label{eq:downlinkintensity}
  m_x(dy)=\lambda_1\frac{\lambda_0}{\lambda_1}\, \1_{\{|y|>|x|\}}\,dy
=\lambda_0 \, \1_{\{|y|>|x|\}}\,dy.
\end{equation}


Next we move on to clarify the meaning in (\ref{eq:sinr_general}) of
the phrase ``emitter at $x_k$ is active''.


\paragraph{Pure-loss model}

The reference model we call pure-loss is the case where failed
transmission are considered lost and links get occupied only due to
new signals arriving. Hence the target link is busy in slot $n$ if and
only if $A_n=1$ and an emitter at $x_k$ is active if and only if
$A_n^k=1$.  The signal to interference and noise ratio in slot $n$ is
therefore
\begin{equation}\label{eq:sinr_pureloss}
  \mbox{SINR}_n(x)= \frac{A_nS_n a(x)}{I_n(x)+\sigma^2},\;\; 
  I_n(x)=\sum_{x_k\in \Phi_x} A_n^k S^k_n a(x_k),
\end{equation}
where it will be shown that the infinite sum random variable
$I_n(x)$ is well-defined. Moreover, $(I_n(x))_{n\ge 1}$ is an
i.i.d.\ sequence. In principle, this is the situation
studied in ref's \dots

\paragraph{Retransmission model}

To study retransmission dynamics we asssociate with the transmitters
of each link a buffer mechanism which keeps record of any failed
signal transmission and re-sends buffered signals at the first
available slot. Let $B_n$ denote the number of buffered packets on the
selected link and $B_n^k$ the buffer size of the interferer at
$x_k\in\Phi_x$. The event marked by the indicator function in
(\ref{eq:sinr_general}) is $A_n^k+B^k_{n-1}\ge 1$, since the emitter
at $x_k$ is active in slot $n$ either if there is a new signal
($A_n^k=1$) or there exists at least one previously buffered signal
($B_{n-1}^k\ge 1$).

Rather than attempting to trace the exact distribution and the
dependency structure arising over time in the collection of buffer
variables, $(B_n,(B_n^k))_{n\ge 1}$, we propose a simplified but
tractable approximation, as follows. Let
\begin{equation}\label{eq:sinr_loss}
  \mbox{SINR}_n(x)= \frac{Q_nS_n a(x)}{\Ic_n(x)+\sigma^2},\;\;
  \Ic_n(x)=\sum_{x_k\in \Phi_x} Q_n^k S^k_n a(x_k),
\end{equation}
where $Q_n\sim \mathrm{Bin}(1,q_n)$ and
$\{Q_n^k\sim\mathrm{Bin}(1,q_n) ,x_k\in \Phi_x\}$ are independent
Bernoulli random variables. The parameter $q_n$, $q_n\ge p$,
introduced here is a mean-field approximation of the collection of
probabilities
$\{\Pb(A_n+B_{n-1}\ge 1),\Pb(A_n^k+B_{n-1}^k\ge 1),x_k\in\Phi_x\}$, to
be specified below in (\ref{eq:pi_recursion}) as a steady-state busy link probability in
slot $n$.
The essence of introducing the Poisson interference field $\Ic_n(x)$
in (\ref{eq:sinr_loss}) as an approximation of $I_n(x)$ in
(\ref{eq:sinr_general}), is that now, given $\Phi_x$, the random
variables $(\mathrm{SINR}_n(x))_{n\ge 1}$ are independent.  Of
course, by taking $q_n=p$, $n\ge 1$, we recover
(\ref{eq:sinr_pureloss}). A preliminary simulation study in \cite{MK2022}
provides some evidence for considering  (\ref{eq:sinr_loss}) a
reasonable approximation of (\ref{eq:sinr_general}). Our approximation
underlying (\ref{eq:sinr_general}) can be compared to \cite{YQ2018}
Assumption 1, and our indicator variables $Q_n^k$ correspond to the
indicators of active states, $\zeta_{x,t}$, in \cite{YQ2018}.

\subsection{Coverage probability}\label{seq:covprob}

The conditional coverage probability associated with
(\ref{eq:sinr_loss}) of a successful transmission in slot $n$ is a
function of $q_n$ that we denote by
\begin{equation}\label{eq:covprobdef} 
V_T(q_n\, ;x,\Phi_x)=\Pb(\mathrm{SINR}_n(x)>T|\Phi_x).
\end{equation}
For now we fix $q=q_n>p$ and take expectations over interferer
locations $\Phi_x$ to obtain the coverage probability given the link
$(0,x)$, as
\[
V_T(q\,;x)=\E[V_T(q\,; x,\Phi_x)],
\]
which is a radial function of $|x|$.
Consequently, the cell-averaged coverage probability is
\[
V_T(q)= \lambda_1 \int_{\R^d} V_T(q\,;
x)e^{-\lambda_1|B(0,1)||x|^d}\,dx.
\]
Next we obtain  the function $V_T(q)$,
$p\le q\le 1$, as a function of fading
intensity $\mu$, strength of noise $\sigma^2$, threshold $T$, and
attenuation $a(x)$. To ensure that the corresponding Poisson integral with
respect to $\Phi_x$ is well-defined we assume that $a$ is integrable
with respect to the intensity measure $m_x$ in
(\ref{eq:downlinkintensity}), in the sense that for every $x$, $|x|>0$,
we have $\int_{|y|\ge |x|}
\frac{\lambda_0\,dy}{1+a(x)/a(y)}<\infty$.

\begin{lemma}
\begin{align*} 
  &  V_T(q\,; x,\Phi_x)=  \\
  & \quad q 
  \exp\Big\{\int_{\R^d}  \ln \Big(1-\frac{q\,T a(y)}{a(x)+T 
    a(y)} \Big)\,\Phi_x(dy)-\frac{\mu T\sigma^2}{a(x)}\Big\},
\end{align*}
\[
  V_T(q\,; x)=  
    q 
  \exp\Big\{-\int_{|y|\ge |x|} \frac{\lambda_0 q\,T a(y)}{a(x)+T 
    a(y)}\,dy-\frac{\mu T\sigma^2}{a(x)}\Big\},
\]
and
\[
 V_T(q) = \int_0^\infty V_T(q\,; v^{1/d}\ell_0)\,e^{-v}\,dv.
\]
\end{lemma}  

\begin{proof}
Taking expectations first over $Q_1$ and then $S_1$ in
(\ref{eq:sinr_loss}), 
\[
  \Pb(\mathrm{SINR}_1(x)>T|\Phi_x)
    =q \,e^{-\mu T\sigma^2/a(x)}\E[e^{-\mu T \Ic_1(x)/a(x)}|\Phi_x]   
\]
Next, putting $\theta_x=\mu T/a(x)$,
\begin{align*}
  \E[e^{-\theta_x \Ic_1(x)}|\Phi_x]
  = \prod_{x_k\in\Phi_x} (1-q(1-\E[e^{-\theta_x S_k a(x_k)}])).
\end{align*}
Since $\E[e^{-\theta_x S_k a(x_k)}]=\mu/(\mu+\theta_x a(x_k))$, 
\[
  \ln  \E[e^{-\theta_x \Ic_1(x)}|\Phi_x]
  =\int_{\R^d}  \ln \Big(1-\frac{q\,\theta_xa(y)}{\mu+\theta_x
    a(y)} \Big)\,\Phi_x(dy).
\]
Summing up, $\ln   V_T(q\,; x,\Phi_x)$ equals
\begin{align*}
    \ln q -\frac{\mu T\sigma^2}{a(x)}
  +\int_{\R^d}  \ln \Big(1-\frac{q\,\theta_xa(y)}{\mu+\theta_x
    a(y)} \Big)\,\Phi_x(dy),
\end{align*}  
which after reinserting $\theta_x$ is the first statement in the
lemma. The second claim is the regular calculation of exponential
moments of a Poisson integral. The representation of $V_T(q)$ in polar
coordinates is then obtained from (\ref{eq:linkdistribution}). 
\end{proof}


\section{Performance measures}

In this section we analyze the performance of the wireless network
averaged over time and averaged over link and interferer positions.
To keep the results concrete and
explicit, 
as a further preparation we evaluate the
coverage probability for a specific class of path loss functions. Then
we study loss and delay at a marked transmission link by
identifying 
the corresponding buffer size as a nonhomogeneous birth-death Markov
chain. The jump probabilities are updated from slot to slot dependent on
the system busy link probabilities, chosen at this point to represent
a mean field dynamics of the entire system.  By studying the
associated nonlinear recursion we may then apply a result from
classical Markov chain theory \cite{IM1976} and verify that the
nonhomogeneous buffer Markov chain is strongly ergodic. 
From this we obtain throughput, loss probability and delay as functions
of the data rate input.

\subsection{Utilization under normalized
  attenuation} \label{seq:utilization} 

From now on we specify the function $a$, which determines the strength of
interaction in the 
SINR-ratio.
First we use the radius
$\ell_0$ appearing in (\ref{eq:linkdistribution}) as a reference
distance and prescribe the normalization property $a(\ell_0)=1$.  This
amounts to the assumption that if we add more base stations per volume
unit, by increasing the parameter $\lambda_1$, then the transmission
protocols are also tuned in order to match the new, on average
smaller, size of the Voronoi cells.  To achieve this setting and stay
close to standard assumptions, we introduce the scaled attenuation
function
\begin{equation}\label{eq:pathloss}
  a_{\kappa,\beta}(x)=\frac{1+\kappa}{|x/\ell_0|^\beta+\kappa},\quad
  x\in \R^d,\quad \beta>d,
\end{equation}
which satisfies the condition for Lemma 1.
Here, $\beta$ is the pathloss exponent and the parameter
$\kappa\ge 0$ controls the close range attenuation,
$a_{\kappa,\beta}(0)=1+1/\kappa$.  The case $\kappa=0$ is tractable
for computations and is frequently used in the literature, but may be
considered nonrealistic because of the singularity of $a_{0,\beta}$ as
$x\to 0$. In relation to this, \cite{AAB2018} discusses ``a broad
class of path loss models that are physically reasonable'',
characterized by the total average received power (for downlink
traffic) being finite. For comparison, we observe that the average
power $P_\mathrm{avg}$ displayed in \cite{AAB2018} corresponds in the
present case to
\[
  P_\mathrm{avg}=\E\Big[ \sum_{x_k\in \Phi} S^k\,
a_{\kappa,\beta}(x_k)\Big ]
=\frac{1}{\mu}\int_0^\infty  \frac{1+\kappa}{v^{\beta/d}+\kappa}\,dv,
\]
which is finite whenever $\kappa>0$ and $\beta>d$.
We reserve the additional notation $\delta$ for the ratio
\[
  \delta=\beta/d>1.
\]
\begin{lemma} \label{lem:covprob_updown}
The coverage probability $V_T(q)$ obtained in Lemma 1
has the 
representation 
\begin{align*}
    V_T(q) = \int_0^\infty q\,
&\exp\Big\{-\frac{\mu
  T\sigma^2(\kappa+v^\delta)}{1+\kappa}\\
&        - q\int_v^\infty \frac{wT(\kappa+v^\delta)}
              {\kappa+u^\delta+T(\kappa+v^\delta)}\,du\Big\}e^{-v}\,dv.
\end{align*}
\end{lemma}

\begin{proof}
We re-write $V_T(q\, ;x)$ in Lemma 1 as  
\begin{align*}
            q e^{-\mu T\sigma^2/a(x)}\exp\Big\{- \frac{\lambda_0 q}{\lambda_1}
        \int_{|x/\ell_0|^d}^\infty \frac{Ta(u^{1/d}\ell_0)}{a(x)+T
             a(u^{1/d}\ell_0)}\,du\Big\},
\end{align*}
after a change to polar coordinates, $u=|y/\ell_0|^d$.  Hence,
changing variable to $v=|x/\ell_0|^d$ and substituting $a$ for $a_{\kappa,\beta}$,
\begin{align*}
  &V_T(q\,; v^{1/d}\ell_0) =\\
           &q \, \exp\Big\{-\frac{\mu T\sigma^2(\kappa+v^\delta)}{1+\kappa}  
           -q \int_v^\infty \frac{wT(\kappa+v^\delta)}
              {\kappa+u^\delta+T(\kappa+v^\delta)}\,du\Big\}.
\end{align*}
The averaged coverage probability $V_T(q)$ is now obtained by
integration over $v$ as specified in Lemma 1.
\end{proof}

\subsection{Loss and delay systems} \label{seq:lossdelay}

Let $Z_n$ be the indicator function of the event that the link is busy and the
transmission is successful in slot $n$,
\begin{equation}\label{eq:success}
  Z_n = \1_{\{\mathrm{SINR}_n(x) >T\}},  \quad n\ge 1.   
\end{equation}
Then $V_T(q_n; x,\Phi_x)=\E[Z_n|\Phi_x]$ and $V_T(q_n)=\E{Z_n}$, by
(\ref{eq:covprobdef}).  Similarly, let $Z_n^+$ be the indicator
function of a successful transmission given that the link is busy in
slot $n$,
\begin{equation*} 
  Z_n^+ = \1_{\{\mathrm{SINR}^+_n(x) >T\}},  \quad
  \mathrm{SINR}^+_n(x) =\frac{S_na(x)}{\sigma^2+\Ic_n(x)}
\end{equation*}
and put $U_n=\E[Z^+_n]$, $n\ge 1$.  Here $Z_n^+$ is independent of
$Q_n$, hence of $A_n$ and $B_{n-1}$. Moreover,
$Z_n\stackrel{d}{=}Q_nZ_n^+$ and $U_n=V_T(q_n)/q_n$. Hence, $U_n$ is
the probability of successful transmission in slot $n$ given that
there is at least one signal to transmit. It remains to identify how
the sequence $(q_n)_{n\ge 0}$ is updated from one slot to the next.

\paragraph{Pure-loss system}

There are no retransmissions carried out by the system and the
signal-to-noise-interference ratio is given by
(\ref{eq:sinr_pureloss}).  Hence the busy link probability is equal to
the arrival probability, $q_n=p$, constant over slots. Letting $L_n$ be
the number of lost signals up to slot $n$, the sequence
$(L_n)_{n\ge 0}$ satisfies
\[
  L_n=L_{n-1}+A_n-Z_n,\quad n\ge 1, \quad L_0=0, 
\]
noting that the difference $A_n-Z_n$ is either zero or one, by
(\ref{eq:sinr_pureloss}). For comparison with the buffer models below
we associate with the pure-loss system an empty buffer of size
$K=0$.

\paragraph{Finite buffer loss and delay system}

For a fixed integer $K\ge 1$ we introduce $(B_n)_{n\ge 0}$ 
by the recursion $B_0=0$ and
\begin{equation}\label{eq:bufferrecursion}
  B_n=B_{n-1}+A_n\cdot (1-\1_{\{B_{n-1}=K\}})-Z_n, \quad n\ge 1,
\end{equation}
and $(L_n)_{n\ge 0}$ by $L_0=0$ and 
\[
 L_n= L_{n-1}+ A_n\cdot \1_{\{B_{n-1}=K\}},  \quad n\ge 1.
\]
Here, $B_n$ is the number of signals stored, hence delayed, in a
physical buffer of maximum buffer capacity $K$ in slot $n$ and $L_n$
is the number of lost signals up to slot $n$. It follows that the
current buffer increases by one unit upon a non-lost arrival and
subsequent transmission failure, while the buffer decreases by one
unit in the event of no arrival or lost arrival combined with the
successful transmission of a buffered signal. Thus, if $B_{n-1}=k$ we
have jump probabilities $\beta_k(n)$ for upward jumps and
$\delta_k(n)$ for downward jumps in slot $n$, given by
\begin{align*}
  &\beta_k(n)=p(1-U_n), \quad 0\le k\le K-1\\
  &\delta_k(n)=(1-p)U_n,\quad 1\le k\le K-1,\quad \delta_K=U_K.
\end{align*}
Also, put $\beta_K(n)=0$ and $\delta_0(n)=0$. We denote by
$M_n=(m_{ij}^n)$ the tri-diagonal $(K+1)\times(K+1)$-matrix with
nonzero elements
\begin{align*} 
  &m^n_{i,i-1}=\delta_i(n),\quad 1\le i\le K\\
  &m_{ii}^n = 1-\beta_i(n)-\delta_i(n),\quad 0\le i\le K\\
    &  m_{i,i+1}^n=\beta_i(n),\quad 0\le i\le K-1.
\end{align*}
For each fixed $n\ge 1$, given $p$ and $q_n$, $0<p<q_n<1$, $M_n$ is
the discrete time transition probability matrix of an irreducible and
aperiodic, homogeneous birth-death Markov chain on
$\{0,\dots,K\}$. The associated unique invariant distribution is the
solution $\nu^n$ of the finite system of linear equations
$\nu^n=\nu^n M_n$. This is a standard example of a Markov chain for
which the stationary distribution is a (variation of) geometric. We
state the result as a Lemma to prepare for analyzing $(B_n)_{n\ge 1}$.
\begin{lemma}\label{lem:buffer_fixn}
  Fix $n\ge 1$, let
\[
    b_n=p(1-U_n)/(1-p)U_n,
\]
and let $\nu^n=(\nu_0^n.\dots,\nu^n_K)$ be the truncated and weighted
geometric probability distribution 
\[
  \nu_k^n=b_n^k\nu_0^n,\quad 0\le k\le K-1,\quad \nu_K^n=(1-p)b_n^K\nu_0^n,
  \]
where
$\displaystyle{
  \nu_0^n =\frac{1-b_n}{1-(p+b_n-pb_n)b_n^K}.}
$
Then $\nu^n=\nu^n M_n$. 
\end{lemma}
The parameter $q_n$ represents the average probability in slot $n$
that a transmission link, target or interfering, is busy. For this we
now propose to use the steady-state distributions $\nu^n$ and update the
busy-link probabilities accordingly, as  $q_{n+1}=1-(1-p)\nu_0^n$.
Therefore, we consider
\[
    \pi^n=(\pi_0^n,\dots,\pi_K^n),\quad \pi_k^n=\Pb(B_n=k),
\]
and define recursively,  given $p$ and $\pi^0$,
\begin{align}  \nonumber
q_n&=1-(1-p)\nu^{n-1}_0,\quad 
U_n=V_T(q_n)/q_n\\
   \pi^n&=\pi^{n-1} M_n. \label{eq:pi_recursion}
\end{align}
In summary, $(B_n)_{n\ge 0}$ is a discrete time, nonhomogeneous,
birth-death Markov chain defined by the family of probability transition
matrices $(M_n)_{n\ge 0}$.

\paragraph{Infinite buffer no loss system}

We extend to the case $K=\infty$ and ask whether it is possible to run
the system in a no-loss mode where all signals are eventually
retransmitted. Thus, we put $L_n=0$ for every $n$, and suppose
that $(B_n)_{n\ge 1}$ satisfies
\[
  B_n=B_{n-1}+A_n-Z_n,\quad n\ge 1,\quad B_0=0.
\]
For this case let us assume $p< U_n$, and let $M_n^\infty$ be the infinite matrix
defined as $M_n$ but 
without terminal state $K$. Then $0<b_n< 1$ and the geometric
distribution 
$\nu_k^n=b_n^k(1-b_n)$, $k\ge 0$, satisfies $\nu^n=\nu^n M_n^\infty$.
For fixed $n\ge 1$, this is the standard birth-death Markov chain
corresponding to a discrete time $M/M/1$ queue with traffic intensity
$b_n<1$.  Now apply (\ref{eq:pi_recursion}) again for the case $K=\infty$.

To summarize the three cases we note that the sum
$D_n=B_n+L_n$, which measures current buffer size plus the accumulated
number of suffered losses in slot $n$, in each case 
defines a sequence $(D_n)_{n\ge 1}$ which satisfies the Lindley-type recursion
\begin{equation}\label{eq:lindleyprel}
  D_n=D_{n-1}+A_n- Z_n,\quad n\ge 1,\quad D_0=0.
\end{equation}

\subsection{The nonhomogeneous Markov chain $(B_n)$}

A nonhomogeneous Markov chain with an associated sequence of
transition matrices $(M_n)$ is said to be strongly
ergodic if there exists a probability distribution $\pi$ such that for
all $m\ge 1$,
\begin{equation}\label{eq:stronglyergodic}
  \lim_{n\to\infty} \sup_{\mu} d_{TV}(\mu M_m\cdots M_n,\pi)=0,
\end{equation}
where $d_{TV}(\mu,\nu)$ is total variation distance between
probability distributions on $\{0,\dots,K\}$ and the supremum is over
all such distributions \cite[Def.\ 8.2]{B1999}.

We now state and prove the main technical result of this paper, which
is the identification of a distribution $\pi$ for which the buffer
sequence $(B_N)_{n\ge 1}$ is strongly ergodic. This is an application
of a central result for nonhomogeneous Markov chains due to Isaacson
and Madsen \cite{IM1976}.  First, we consider the family of probability
distributions $(\nu^n)_{n\ge 1}$ in Lemma \ref{lem:buffer_fixn} linked
via the sequence $(q_n)_{n\ge 1}$ in (\ref{eq:pi_recursion}).

\begin{lemma} \label{lem:q_seq}

For given $p$, $0<p<1$, take $q_0=p$.
Fix $K\ge 1$. Then $q_n= F(q_{n-1})$, $n\ge 1$, where
$F(q)=1-G(U(q))$, $p\le q\le 1$, and
  \[
    G(u)=\frac{u-p}{u-p\big(p(1-u)/(1-p)u\big)^K}, \quad
  U(q)=\frac{V_T(q)}{q}.
\]
As $n\to\infty$,  $q_n\to q^*$, 
where $q^*=q^*(p)$, $p<q^*<1$, is the minimal solution of the equation
$q=F(q)$.

For $K=\infty$ we have $q_n= p/U(q_{n-1})$. For
every $p$ such that the balance eqution $V_T(q)=p$ has at least one
solution then $\lim_{n\to\infty}q_n=q^*$, where $q^*=q^*(p)$ is the
minimal solution.
\end{lemma}

\begin{proof} For the case $K<\infty$, with $q_0=p$ we put
\[
  \nu_0^0 =\frac{1-b_0}{1-(p+b_0-pb_0)b_0^K},\;
  b_0=\frac{p(1-U_0)}{(1-p)U_0},\; U_0=\frac{V_T(p)}{p}. 
\]
Then, rewriting, $q_1=1-(1-p)\nu_0^0=F(q_0)$. In the same manner
$q_n=F(q_{n-1})$ with $F$ as stated in the Lemma.  It is clear from
Lemma \ref{lem:covprob_updown} that $V_T(q)$ is differentiable and
$qV_T'(q)<V_T(q)$. Hence $U(q)<1$ is a decreasing function of $q$.
Moreover, $G(u)$ is an increasing function of $u\in [0,1]$. Thus,
$F(q)$, $p\le q<1$, is a continuous function which is increasing from
$F(p)=1-G(U(p))>1-G(1)=p$ at the left end point of the interval to
$F(1)=1-G(U(1))<1$ at the right end point. Therefore there must be a
unique, minimal stationary point $q^*\in (p,1)$ of $q=F(q)$, such
that $F'(q^*)<1$.  Hence $\{q_n\}_{n\ge 1}$ converges and the limit
is $q^*$.

The infinite buffer case $K=\infty$ has $F(p)>p$ but now
$F(1)=p/U(1)$ may exceed $1$. Hence in this case we need to find a
minimal solution $q^*$ of $q=F(q)$, i.e.\ $V_T(q^*)=p$, and then
check that $V_T'(q^*)>0$ to verify $F'(q)|_{q=q^*}<1$.
\end{proof}

For given data rate $p$ we can now put $U^*=V_T(q^*)/q^*$ and let
$M^*=M^*(p)$ be the probability transition matrix of the homogeneous
birth-death Markov chain with jump probabilities $\beta_k=p(1-U^*)$,
$0\le k\le K-1$, for upward jumps and $\delta_k=(1-p)U^*$,
$1\le k\le K-1$, $\delta_K=U^*$, for downward jumps.




\begin{theorem} \label{thm:main}
The nonhomogeneous Markov chain $(B_n)_{n\ge 0}$ defined in
(\ref{eq:pi_recursion}) is strongly ergodic such that
(\ref{eq:stronglyergodic}) holds for the distribution $\pi$ given by 
\[
  \pi_k=b^k \pi_0 , \;0\le k\le K-1, \;\;\pi_K=(1-p)b^K\pi_0
\]
and $\pi_0=(1-q^*)/(1-p)$, $b=p(1-U^*)/(1-p)U^*$.
\end{theorem}

\begin{proof}

  Observing that $\|A\|=\sup_{0\le i\le K} \sum_{j=0}^Ka_{ij}$ is a
  norm on the vector space of square matrices $A=(a_{ij})$, we
  conclude from Lemma \ref{lem:q_seq} that $\|M_n-M^*\|\to 0$,
  $n\to\infty$.  Moreover, $M^*$ is the transition matrix of an
  ergodic Markov chain.  These properties of $(M_n)$ are exactly the
  sufficent conditions for strong ergodicity obtained in \cite[Theorem
  V.4.5]{IM1976}, also quoted in \cite[Theorem 6.8.5, without
  proof]{B1999}. The corresponding limit distribution $\pi$ is the
  unique invariant distribution of $M^*$, for which $\pi=\pi
  M^*$. It is straightforward to check the form of $\pi_k$, noticing
  that the normalization $\pi_0=(1-q^*)/(1-p)$ is the relation
  $q^*=F(q^*)$ in Lemma \ref{lem:q_seq}.
\end{proof}


\subsection{Time-averaged throughput} \label{seq:time-average}

To analyze the delay-loss systems we first
have, by (\ref{eq:lindleyprel}),
\[
  \frac{B_n}{n} +\frac{L_n}{n}
  =\frac{1}{n} \sum_{r=1}^n (A_r - Z_r).
\]
The rightmost representation is a sum of
independent random variables with expected value
\[
  \E\Big[\frac{B_n+L_n}{n}\Big] 
  =p- \frac{1}{n} \sum_{r=1}^n V_T(q_r).  
\]
Thus, using Lemma \ref{lem:q_seq} for the case $K\ge 1$,
by the Kolmogorov strong law of large numbers and the continuity of $V_T$
in $q$, we have the almost sure convergence as $n\to\infty$,
\begin{equation}\label{eq:limit_Dn} 
  n^{-1}\E[B_n+L_n]\longrightarrow p-V_T(q^*).
\end{equation}
To interpret the above limit we introduce the {\it loss probability}
\[
  P_\mathrm{loss}(p,q)=1-V_T(q)/p,\quad 0\le p\le q\le 1.
\]
In the case of pure-loss, $K=0$, the convergence (\ref{eq:limit_Dn})
establishes the loss rate
\begin{equation}\label{eq:purelossprob}
  n^{-1}\E[L_n]\to p-V_T(p)=p P_\mathrm{loss}(p,p).
\end{equation}
Similarly, with a finite buffer $1\le K<\infty$ we have
$B_n\le K$ so $B_n/n\to 0$ as $n\to\infty$, and hence the asymptotic loss rate
\[
  n^{-1}\E[L_n]\to p P_\mathrm{loss}(p,q^*).
\]
To extract additional information we may analyze the buffer sequence
in (\ref{eq:bufferrecursion}) alone, using
\[
  \frac{B_n}{n} 
 = \frac{1}{n}\sum_{r=1}^n \{ A_r\cdot (1-\1_{\{B_{r-1}=K\}}) -Z_r\}.
\]
Again, $B_n/n\to 0$ as $n\to\infty$ since $B_n\le K$. Hence by the
law of large numbers, 
  $\lim_{n\to\infty} p (1-\Pb(B_n=K))-V_T(q^*)=0$,
and so 
\begin{equation}\label{eq:balance_loss}
  \lim_{n\to\infty} \Pb(B_n=K)=1-V_T(q^*)/p=P_\mathrm{loss}(p,q^*),
\end{equation}
which confirms that losses in the system are only incurred during
slots at which $B_n=K$, and also that $V_T(q^*)$ is the {\it
  throughput} of the system in the sense measured by the limit of
$p(1-\Pb(B_n=K))$, see e.g.\ \cite{K2002} for general remarks on
performance measures.

The corresponding result for the infinite buffer
no loss model is $B_n/n\to p-V_T(q^*)$, and hence to maintain
the retransmission system in steady-state without losses we must find
a solution $(p,q^*)$ of the {\it balance equation} $p-V_T(q^*)=0$.






\begin{corollary} \label{cor:kpi}

1) Fix an integer buffer size $K\ge 1$. For given data rate $p$,
$0<p<1$, let $q^*=q^*(p)$ be the limiting busy-link probability introduced in
Lemma \ref{lem:q_seq}.
Then
the throughput equals
\[
  p(1-\pi_K)=V_T(q^*),
\]
the loss probability is
\[
  P_\mathrm{loss}(p,q^*)=\pi_K=1-\frac{V_T(q^*)}{p},
\]
and the average delay, or latency, in the model, obtained by
a standard application of Littles formula as
$D(p)=\E_\pi[B]/p$, is 
\begin{equation} \label{eq:delay_thm}
D(p)=\frac{V_T(q^*)(q^*-V_T(q^*))-q^*K(p-V_T(q^*))}{p(V_T(q^*)-pq^*)}.
\end{equation}

\noindent
2) Let $p_c$ be such that for every $p<p_c$, the non-linear equation
$V_T(q)=p$ has a minimal solution $q^*=q^*(p)$, $p<q^*<1$. Then the infinite
buffer distribution arising in the limit $K\to\infty$ is the regular
geometric distribution
\[
  \pi_k = \Big(\frac{q^*-p}{1-p}\Big)^k\frac{1-q^*}{1-p}, \quad k\ge 0,
\]
which has throughput $p$ and latency $D(p)=(q^*-p)/(1-q^*)p$.
\end{corollary}

\begin{proof}

Using Theorem \ref{thm:main} together with
(\ref{eq:balance_loss}), 
\begin{equation}\label{eq:balance_K}
  p(1-\pi_K)
  =p \cdot \frac{1-q^*}{1-p}
 \cdot \frac{1-b^K}{1-b}=V_T(q^*),
\end{equation}
To obtain the latency we use
\begin{equation*}\label{eq:delay}
  \E_\pi[B]=\frac{\pi_0}{1-b}\Big\{b\frac{1-b^K}{1-b}-Kb^K (p+(1-p)b)\Big\}
\end{equation*}
and apply
$\pi_K=(1-p)b^K\pi_0=1-V_T(q^*)/p$ and (\ref{eq:balance_K})
to recognize
\begin{align*}
  \E_\pi[B]&=\frac{1}{1-b}\Big\{b
             \frac{V_T(q^*)}{p}-K\frac{1-V_T(q^*)/p}{1-p}(p+(1-p)b)\Big\}.
\end{align*}
The representation (\ref{eq:delay_thm}) follows after the final step
of inserting the relation $b(1-p)V_T(q^*)=p(q^*-V_T(q^*))$.
\end{proof}

\section{Analysis of some cases of the traffic models}

For the pure loss case $K=0$, since the busy link probability $q$
coincides with the arrival probability $p$, the coverage
probability obtained in Lemma \ref{lem:covprob_updown} is ready
to use.  The loss rate is derived in (\ref{eq:purelossprob}).
For finite $K\ge 1$, a natural procedure is to start with the coverage
probability function $V_T$ and for each $p$ compute iteratively busy
link probabilities $q_n$, throughput $V_T(q_n)$, loss rate
$P_\mathrm{loss}(p,q_n)$, etc, for $n\ge 1$, $q_0=p$, with $n$ large
enough until the sequences converge within limits of a preset
numerical accuracy.  This yields approximations of the key performance
indicators obtained in Corollary \ref{cor:kpi} as functions of
data rate $p$, $0< p<1$.

For special choices of parameter values the results may simplify.  Let
us introduce the function
\[
  K_\delta(r)=\int_r^\infty  \frac{du}{1+u^\delta},\quad r\ge 0,
\]
and put
\[
  C=C_{T,\delta,w}= T^{1/\delta}K_\delta(1/T^{1/\delta}) w,\quad w=\lambda_0/\lambda_1.
\]
Using Lemma \ref{lem:covprob_updown}
with parameters $\kappa=0$ and $\sigma^2>0$, we recover the result in
Andrews, Baccelli, Ganti 2011, Theorem 1, namely
\begin{equation}\label{eq:explicitV}
V_T(q)=q \int_0^\infty e^{-qCv}  
e^{-\mu\sigma^2 T v^\delta} e^{-v}\,dv.
\end{equation}
In particular, for $\kappa=\sigma^2=0$,
\begin{equation}\label{eq:Vexplicit}
V_T(q)=\frac{q}{1+Cq}=\frac{q}{1+T^{1/\delta}K_\delta(1/T^{1/\delta})w\,q}.  
\end{equation}
A simple example is pure loss, $K=0$, with $\kappa=\sigma^2=0$, for
which 
the asymptotic loss rate is $n^{-1}L_n\to  Cp^2/(1+Cp)$.

\subsection{Finite buffer, the case $K=1$}\label{sec:K1}

\begin{figure}
 \includegraphics[width=0.5\textwidth]{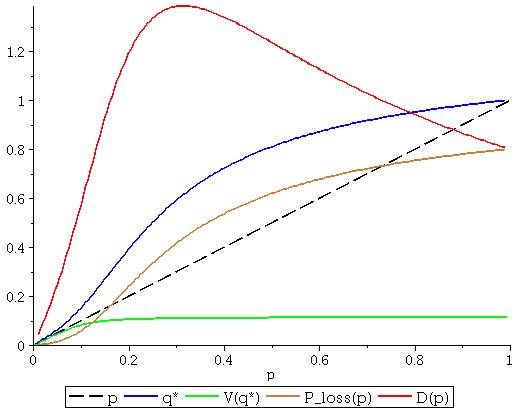}

\caption{The case $K=1$, $\kappa=\sigma^2=0$, $C=4$. 
Data rate  $p$, busy-link probability $q^*$, throughput $V_T(q^*)$,
loss probability $P_\mathrm{loss}(p,q^*)$, latency $D(p)$ (gold).
}
\label{fig:K1new}
\end{figure}


By Lemma \ref{lem:q_seq}, for $K=1$ and $0<p<1$, the limiting busy link
probability $q=q^*$  is the minimal solution of the balance equation 
$V_T(q)=p(1-q)/(1-p)$ 
and then the buffer size distribution is obtained as
$\pi_0=(1-q)/(1-p)$, $\pi_1=(q-p)/(1-p)$.
With the coverage probability in (\ref{eq:Vexplicit})
this is a second order equation with solution
\[
  q^*=(\sqrt{(1-Cp)^2+4Cp^2}-(1-Cp))/2Cp,
\]
hence the throughput as a function of $p$ equals
\[
  V_T(q^*)=\frac{1}{2C(1-p)}\Big(1+Cp-\sqrt{(1-Cp)^2+4Cp^2}\Big). 
\]
The delay is
\[
  D(p)=\E_\pi[B]/p=\pi_1/p=(q^*-p)/p(1-p),
\]
which may be compared with the loss probability
$P_\mathrm{loss}(p,q^*)=1-V_T(q^*)/p$.  Figure \ref{fig:K1new} for
$K=1$ visualizes the busy link probability, the coverage probability,
the latency, and the loss probability as functions of $p$, for fixed
$T$, $\delta$ and $w$ such that $C_{T,\delta,w}=4$. The latency
represents time in buffer per arrival of new data. Hence the decrease
in latency $D(p)$ for large data rates $p$ does not reflect improved
performance but is merely a consequence of larger losses while
throughput quickly reaches the maximum level $V_T(1)=1/5$.


\subsection{No loss, the case $K=\infty$}

Again we study the explicit situation in (\ref{eq:Vexplicit})
with $\kappa=0$ and $\sigma^2=0$.
For the no-loss model with given $w$ and $\delta$, and
a fixed threshold $T$, there is a critical probability $p_c=p_c(T)$ such that
for $0<p\le p_c$, the balance equation, $V_T(q)=p$, has a uniqe solution
$q^*$, $p<q^*\le 1$, given by
\[
  q^*=\frac{p}{1-C p},\quad p\le p_c=\frac{1}{1+C}.
\]
Equivalently, for fixed $p$ there is a critical threshold parameter
$T_\mathrm{max}$ such that, if 
\[
  T\le T_\mathrm{max}=\sup\Big\{y:
  yK_\delta(1/y)<\frac{1}{w}\Big(\frac{1}{p}-1\Big)\Big\}^\delta,
\]
then again the balance equation has the unique solution $q^*=p/(1-C p)$.
The range of $T$-values is well-defined since the function
$y\mapsto yK_\delta(1/y)$ is increasing. The buffer size has the
geometric distribution
\[
  \Pb(B=k)=\pi_k=b^k(1-b), \; k\ge 0, \quad
  b=\frac{Cp^2}{(1-Cp)(1-p)}
\]
and the delay is 
\[
D(p)= \E_\pi(B)/p=Cp/(1-p(1+C)),\quad p<p_c.
\]


\begin{figure}
  \includegraphics[width=0.45\textwidth]{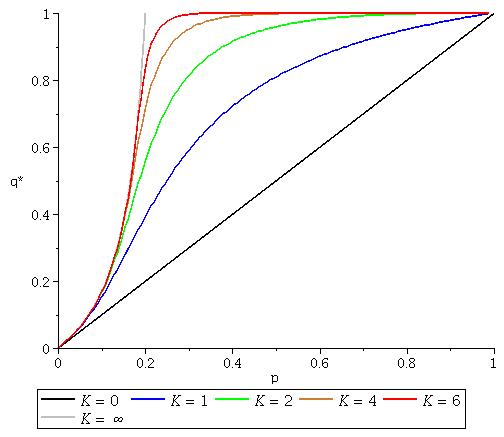}   
\includegraphics[width=0.45\textwidth]{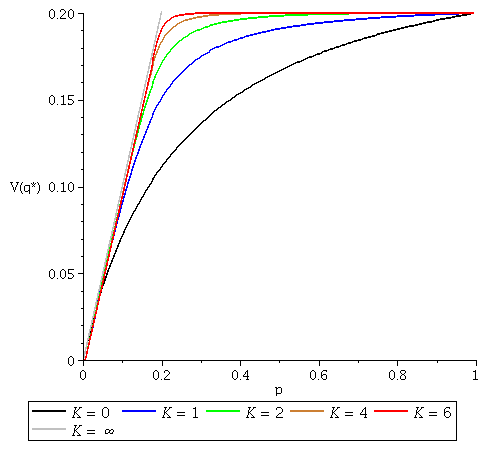} 

\caption{The case $\kappa=0$, $\sigma^2=0$. Critical arrival
  probability $p_c=1/5$. Upper panel: Busy link probabability $q^*$ as
  function of $p$ for varying buffer size $K$. Lower panel:
  Coverage probability $V_T(q^*)$ for varying $K$.}
\label{fig:QV}
\end{figure}

\subsection{Loss and delay as functions of buffer size $K$, $\kappa=\sigma^2=0$}

To study how the loss and delay indicators change with buffer
size $K$ it is convenient to continue with 
the reference case $\kappa=\sigma^2=0$.
We have discussed the solutions $q^*$ and $V(q^*)$ for
$K=0$, $K=1$, and $K=\infty$.
For $K=2,4,8$, additional numerical solutions are
displayed graphically in Figure \ref{fig:QV}. The parameter $C=4$ is
an arbitrary choice.


The corresponding loss probability $P_\mathrm{loss}(p,q^*)$ and the buffer delay
$D(p)$ are  shown together in Figure \ref{fig:LDlog}, upper panel.  To reveal
more detail, these graphs  have a logarithmic scale and
focus on the most relevant parameter values $p$ not too far from
$p_c=1/6$, which is the upper bound for the infinite buffer case.
\begin{figure}
  \quad
  \includegraphics[width=0.45\textwidth]{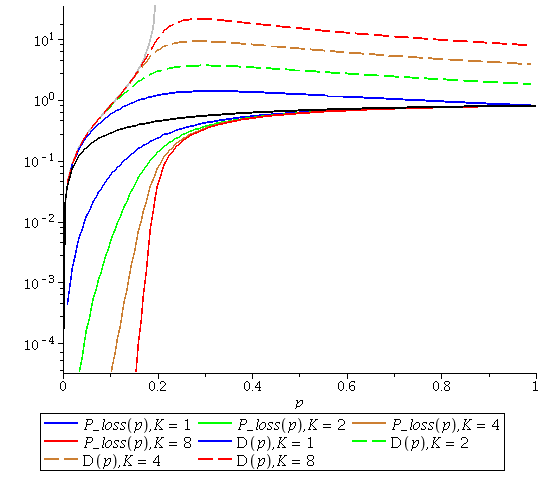} 

  \vskip 3mm
  \includegraphics[width=0.45\textwidth]{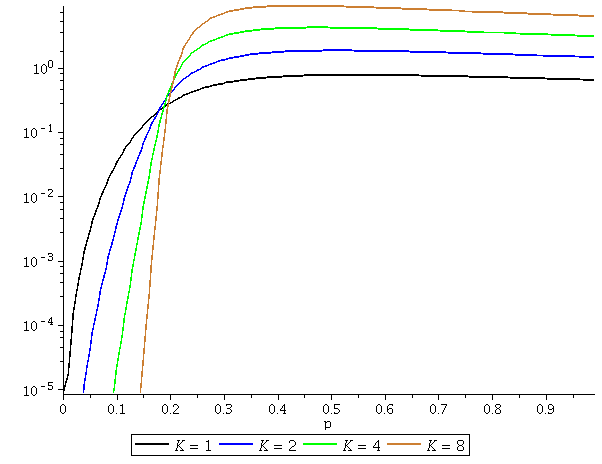} 
\caption{Upper panel: Logarithmic plots of loss probability (solid
  lines) and delay (dashed lines).
  Lower panel:  Logarithmic plots of the loss-delay product.
}

\label{fig:LDlog}
\end{figure}
The graphs confirm intuition that large buffer size favors reliability
at the cost of latency and lead us to consider the trade-off between
reliability and latency in the system subject to a given data rate.
For $p$ well below $p_c$ it is seen that the loss probability is
decreasing and hence the system is increasingly reliable with larger
$K$, while latency as encoded by the buffer delay is still kept at
reasonable levels. For $p$ near or above $p_c$, on the other hand, the
delay builds up with increasing $K$ while the gain in smaller loss
becomes less pronounced.  These insights suggest that the
loss-delay product
\[
  LD(p)=(1-V_T(p,q^*)/p)\cdot D(p),\quad 0<p<1,
\]
is a potentially relevant performance measure. Indeed, Figure
\ref{fig:LDlog}, lower panel,  displayed in logarithmic scale indicates a clear
shift in $LD$-performance around the critical $p_c$. Large buffer size
is preferred at lower traffic intensities but quickly turns
counter-productive as soon as the system is subject to heavier data
rates $p$ ranging above $p_c$.

In principle it is now possible to analyze a network with specified
loss and delay requirements. Given network parameters $w$, $T$ and
$\delta$, suppose we fix a maximal loss probability $L_\mathrm{max}$
and maximal average delay $D_\mathrm{max}$.  Assuming we know that the
network is operating at maximal data rate $p_0$, find for which buffer
size values $K$ it holds that
$P_\mathrm{loss}(p,q^*)\le L_\mathrm{max}$ and
$D(p)\le D_\mathrm{max}$, for all $p\le p_0$. If there is a range of
such $K$-values we are free to pick a prefered $K$ perhaps based on
additional priorities for loss versus delay. If there is no such $K$
the desired loss-delay requirements are not viable for the system.
Figure \ref{fig:LDmax} provides a display. For example, with
$p_0=0.09$ then for $D_\mathrm{max}=6$ we find $K<10$ and for
$L_\mathrm{max}=0.03$ we find $K>3$. So for $4\le K\le 9$ both of the
loss and delay criteria are fulfilled for every $p<p_0=0.09$.

\begin{figure}
\includegraphics[width=0.45\textwidth]{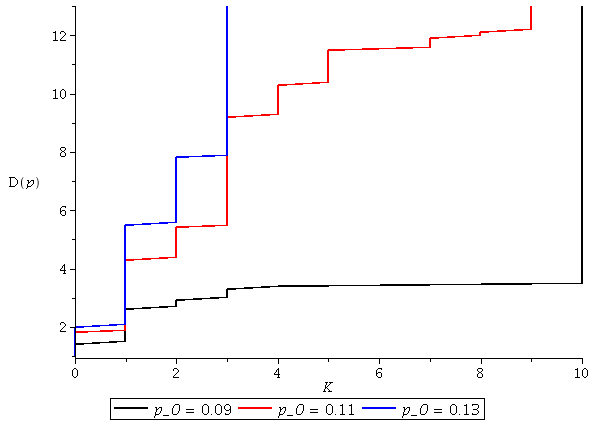}
\includegraphics[width=0.45\textwidth]{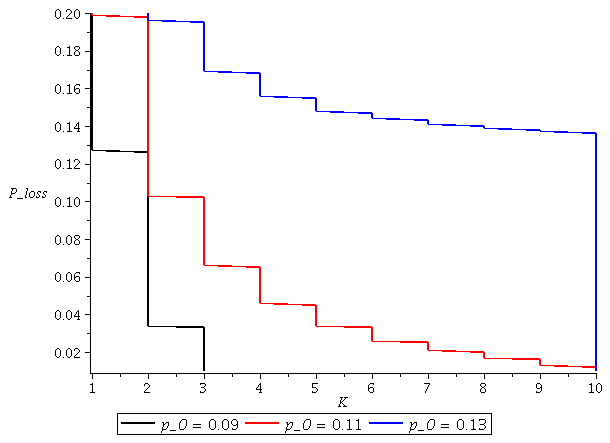}

\caption{Choice of buffer size $K$ to
  provide preset loss and delay criteria
}
\label{fig:LDmax}
\end{figure}

\subsection{More realistic scenarios, $\sigma^2>0$ and/or $\kappa>0$}

Moving away from the simplest coverage probability function $V_T(q)=q/(1+Cq)$,
we consider briefly the case (\ref{eq:explicitV}) with unbounded path loss
($\kappa=0$) but nonzero deterministic noise $\sigma^2>0$ and the  
general coverage probability for bounded path loss function,
$\kappa>0$, derived in Lemma \ref{lem:covprob_updown}.
The display in Figure \ref{fig:Vofq_general} shows the shape of
$V_T(q)$ as a function of the busy link probability $q$, for various
combinations of $\sigma^2$ and $\kappa$, and other parameters fixed at
$\mu=T=1$, $\delta=2$, $w=10$.  The feature that stands out is that
for $\kappa>0$, $V_T(q)$ may no longer be an increasing function of
$q$. As a consequence the balance equation (\ref{eq:balance_K}) may
have more than one solution, noting that the properties in Theorem
\ref{thm:main} refer to the minimal solution.

\begin{figure}
\includegraphics[width=0.45\textwidth]{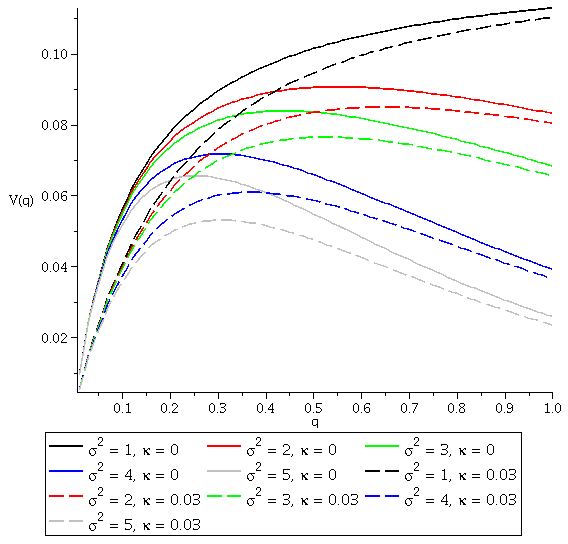} 
\includegraphics[width=0.45\textwidth]{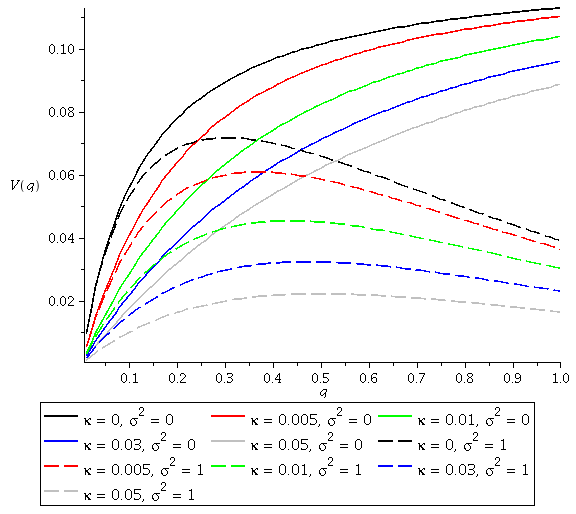}
\caption{The coverage probability function $V_T(q)$ for various
  choices of parameters $\kappa$ and $\sigma$.
}
\label{fig:Vofq_general}
\end{figure}

Finally, given $K$, solving iteratively for a minimal solution $q^*$,
we obtain an approximation of the throughput $V_T(q^*)$ as a function
of the external data rate $p$. The throughput is reduced in comparison
with the unbounded case $\kappa=0$, as indicated in Figure
\ref{fig:Vofp_general}.

\begin{figure} 
  \includegraphics[width=0.45\textwidth]{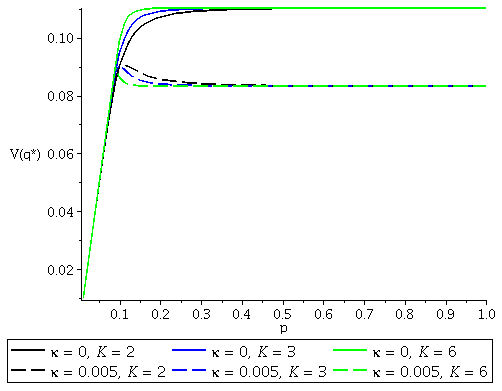}  
\caption{Comparison of throughput, singular attenuation $\kappa=0$,
  bounded attenuation
  $\kappa=0.005$. 
}
\label{fig:Vofp_general}
\end{figure}

\end{document}